\documentclass{amsart}
\pdfoutput=1
\usepackage{amsfonts,amsthm,,amsmath,amssymb,amscd,amsmath}
\usepackage{graphicx}
\usepackage{color, xifthen}

\def\pedantic{1}
\def\pedant#1{\ifthenelse{\pedantic=1}{{ #1}}{}}

\def\Real{{\mathbb R}}
\def\R{{\mathbb R}}

\def\Int{{\mathbb Z}}
\def\d{{\partial}}
\def\vx{{\vec{x}}}

\newtheorem{theorem}{Theorem}[section]
\newtheorem{corollary}[theorem]{Corollary}
\newtheorem{lemma}[theorem]{Lemma}

\newtheorem{theo+}{Theorem}
\newtheorem{prop+}{Proposition}
\newtheorem{coro+}{Corollary}
\newtheorem{lemm+}{Lemma}

\theoremstyle{definition}
\newtheorem{defi+}{Definition}

\theoremstyle{remark}
\newtheorem{rema+}{Remark}

\def\body{{\mathcal B}}

\def\conf{{\mathtt{Conf}}}
\def\config{\conf}
\def\confdb{{\conf}}  

\def\tauf{{\mathbf \tau}}

\begin{document}

\title{Min-type Morse theory for configuration spaces of hard spheres}

\author{Yuliy Baryshnikov}
\address{Departments of Mathematics and ECE, UIUC, Urbana, IL}
\email{ymb@uiuc.edu}

\author{Peter Bubenik}
\address{Department of Mathematics, Cleveland State University}
\email{p.bubenik@csuohio.edu}

\author{Matthew Kahle}
\address{School of Mathematics, Institute for Advanced Study, Princeton NJ 08540}
\email{mkahle@math.ias.edu}
\thanks{The third author thanks IAS and NSA Grant \# H98230-10-1-0227}

\dedicatory{We dedicate this paper to the memory of Boris Lubachevsky.}
\maketitle

\begin{abstract} In this paper we study configuration spaces of
  hard spheres in a bounded region.  We develop a general Morse-theoretic framework
  and show that mechanically balanced configurations play the role of critical points.  As an
  application, we find the precise threshold radius for a configuration space to be homotopy equivalent to the configuration
  space of points.
\end{abstract}


\section{Introduction} 

Configuration spaces of $n$ points in $\Real^d$ are well studied \cite{Cohen}. In this article we are interested in a natural generalization, configuration spaces of non-overlapping balls in a bounded region in $\Real^d$. 

Besides their intrinsic mathematical interest, the study of these spaces is motivated by physical considerations.  For example, in statistical mechanics ``hard spheres'' (or in two dimensions ``hard disks'') are among the most well-studied models of matter.  Computer simulations suggest a solid-liquid phase transition for hard spheres \cite{fun}, but this is not well understood mathematically.

A number of papers in statistical mechanics have explored the hypothesis that underpinning phase transitions are changes in the topology of the underlying configuration space or equipotential submanifolds \cite{Teix, Grinza, exact,  FF}.  Franzosi, Pettini, and Spinelli show that under fairly general conditions (smooth, finite-range, confining potentials), the Helmholtz free energy cannot pass through a phase transition unless there is a change in the topology of the underlying configuration space \cite{FPS1,FPS2}.  This theorem unfortunately does not apply to configuration spaces of hard spheres, since the potential function is not smooth --- but the Morse-theoretic methods developed here may be a step in the direction of extending it to include hard spheres.  

Several other papers have investigated configuration spaces as models of motion planning for robots \cite{Farber, Ghrist}. For example, Farber's ``topological complexity'' can be thought of as measuring the difficulty of designing an algorithm for navigating the space.  As Deeley recently pointed out when he studied ``thick particles'' on metric graphs, the assumption that robots are points is not physically realistic, and giving the points thickness wildly complicates the topology of the underlying configuration space \cite{Deeley}. 

Let $\body$ be a bounded region in $\Real^d$.
Define $\confdb(n,r)$ to be the configuration space of $n$ non-overlapping
balls of radius $r$ in $\body$. We are especially interested here in understanding when the topology changes if $n$ is fixed and $r$ is varying .
First we consider the extreme cases.  For $r$ sufficiently small, one expects that  $\confdb(n,r)$ 
is homotopy equivalent to the configuration space $\conf(n)$ of
$n$ distinct points in $\body$ --- for a survey of configuration spaces of points see Cohen \cite{Cohen}.  On the other hand for $r$
sufficiently large, $\confdb(n,r)$ is empty.  Indeed finding
the smallest such $r$ is the sphere
packing problem in a bounded region --- see for example Graham et al.\
\cite{Graham1,Graham2,Graham3, Graham4} and Melissen \cite{Mel1,Mel2}.

In this note we develop a Morse-theoretic framework which provides a necessary condition for the topology to change --- mechanical balanced configurations play the role of critical points (and submanifolds).  As an illustration of the method, we find the precise threshold radius below which $\config(n,r)$ is homotopy equivalent to $\conf(n)$. 

%


%
%

\section{Tautological Morse function}
Fix $n$, and define $\conf(n)$ to be the set of ordered $n$-tuples of
distinct points in a bounded domain $\body \subset \Real^d$:  $$\conf(n) =
\{ \vx=(x_1,\ldots, x_n ) \mid x_i \in \body , x_i \neq x_j \text{ for } i \neq j \}.$$

As an open subset of $\Real^{dn}$, $\conf(n)$ has the structure of a
smooth manifold.  Let $\tauf : \conf(n) \to \Real$ be defined by
\begin{equation}\label{taut}
\tauf(\vx):=\min \left( \frac{1}{2}\min_{i\neq j} d(x_i,x_j), \min_i
  \min_{p\in \partial\body} d(x_i, p) \right),
\end{equation}
where $\partial \body$ denotes the boundary of $\body$.
We call $\tau$ the tautological function.
Then by definition the configuration space of $n$ balls of radius $r$
in $\body$ is given by
\[
\confdb(n,r)=\tauf^{-1}[r,\infty).
\]
This observation suggests using a ``Morse''-type theory of $\tauf$ 
to study the topology of $\confdb(n,r)$ and especially how the topology changes as
$r$ varies. 
One obvious trouble on that route is the fact that $\tauf$ typically
is not smooth, so that we need a general framework which
allows us to work with non-smooth functions.

In the next section we will discuss the properties of {\em
  min}-type functions, that is functions on a manifold $M$ given
as the minimum of a parametric family
of real valued functions
$$
\tauf(x):=\min_p f(p,x), x\in M, p\in P,
$$
where $P$ is a compact parameter space, and $f$ is continuously differentiable in $x$ for every fixed $p \in P$. 
We note that the function given by
(\ref{taut}) falls within this category, if one considers as $P$ the
disjoint union of the discrete set corresponding to pairs $(i, j),
1\leq i<j\leq n$ and of $n$ copies of the boundary, formed by the
pairs $(i, p), 1\leq i \leq n, p\in\partial\body$.

It should be remarked that the Morse-type theory of the {\em min} (or
even {\em min-max}) type functions has appeared in the literature
(compare \cite{matov, bryzgalova, min-type}), but in a 
much more
restrictive context (with essentially finite parameter space).

\section{Min-type Morse theory}
Let us start with some notation and definitions.

For a manifold $M$ and function $f: M \to \Real$, let $M^c$ denote the
superlevel set at $c$, i.e. $M^c = f^{-1} [c, \infty)$.

We say that a function $h: (s,t)\to \Real$, is {\em increasing with
  speed at least $v>0$} if $h(t')-h(s')\geq v(t'-s')$ for any $s'<t'$ in the
interval $(s,t)$. We note that such a function does not have to be even continuous.

We record for later use the following (immediate) result:
\begin{lemma}\label{lem:inverse}
Let $f:M\times\Real\to\Real$ be a continuous proper function, 
(strictly) 
increasing
along each fiber $\{x\}\times\Real$. Then the fiber-wise inverse
$\phi:(x,c)\mapsto \inf(t: f(x,t)=c)$ is continuous on $M \times [c_1,c_2]$ for any
interval $[c_1, c_2]$ which belongs to the ranges of all the functions
$f_x(\cdot):=f(x,\cdot)$. 
\end{lemma}

\pedant{
\begin{proof}
As $f$ is proper, then the $t$-projection of the 
preimage of a compact interval $[c_1, c_2]$ is compact as well. Further,
if a sequence $(x_i, c_i)$ converges to $(x_*, c_*)$, yet
$t_i:=\phi(x_i, c_i)$ fails to converge to $t_*:=\phi(x_*, c_*)$, we can, using
the boundedness of the sequence $(t_i)$,
choose a subsequence such
that along it $t_j\to t\neq t_*$. By continuity, $\lim f(x_j,
t_j)=f(x_*, t)$ (as $x_j$ converge to $x_*$).
The fact that $f(x_j,
t_j)=c_j$ implies that $f(x_*, t)$ equals $c_*$, which by the continuity of $f$
also equals $f(x_*, t_*)$.
This contradicts the assumption that $f$ increases fiber-wise.
\end{proof}
}

For a smooth vector field $V$ on $M$ we will denote the 
time $t$ shift along
the trajectories of $V$ as $S_t^V$. 
We will say that the function $f$
{\em increases along the trajectories of $V$ with non-zero speed}, if for
some common $v>0$, and for all $x\in M$, $h_x: t\mapsto f(S^V_t x)$
increases with speed at least $v$. 

\begin{lemma}\label{Morse}
  Let $M$ be a smooth manifold and $f: M \to \Real$ a continuous
  function, such that $M^a$ is compact.  Suppose that $M$
  admits a 
  smooth vector field $V$  non-vanishing on $f^{-1}[a,b]$,  
  and such that $f$ is increasing along the trajectories of $V$ 
  on the set $f^{-1}[a,b]$ with non-zero speed. 
  Then $M^b$ is a deformation retract of $M^a$.
\end{lemma}

Lemma \ref{Morse} can be seen as a generalization of Theorem 3.1 in
\cite{Milnor} for non-smooth functions $f$. We remark here that one
can drop here the non-zero speed condition, requiring only that 
$f$ is increasing along the trajectories of $V$, but we do not need
this strengthened form in this paper.

\begin{proof}
For 
$x\in M^a$ set a partially defined function on $M^a\times \Real$
by $g(x,t):= f(S_t^Vx)$. The non-zero speed condition implies that $[a,b]$
is in the range of $g(x, \cdot)$ for any $x\in f^{-1}([a,b])$, and
together with the
compactness of $M^a$ implies that $g$ is proper. Hence, by Lemma~\ref{lem:inverse},
$\tilde\phi(x,c)=\inf(t:f(S_t^V
x)\geq c)$ is well defined and continuous, as well as $\phi(x,c)=\max(\tilde\phi(x,c),
0)$. As for any $c\leq f(x)$, $\phi(x,c)=0$, $\phi$ vanishes on $M^b\times [a,b]$. 

Now define the homotopy
$$
H:M^a\times [0,1]\to M^a
$$
as
$$
H:(x,\tau)\mapsto S^V_{\phi(x, (1-\tau)a + \tau b)}x.
$$
Continuity of $\phi$ implies continuity of $H$; the facts that
$H(x,0)=\mathit{id}_{M^a}$, 
$H(x,\tau)|_{M^b}=\mathit{id}_{M^b}$ for $0 \leq \tau \leq 1$ and $H(x, 1)\in M^b$ are immediate.
\end{proof}

\subsection{Regular values of {\em min}-type functions}

Next we will use the fact that the
tautological function $\tauf$ is the minimum of
a compact family of smooth functions. We want to establish conditions
when a value $c$ is {\em topologically regular}, that is for which there exists some $\epsilon > 0$
 such that $M^{c+\epsilon}$ is a deformation retract of $M^{c - \epsilon}$. 


We give a general condition for
topological regularity, as follows.

Let $P$ be a compact metric space, $M$ a compact smooth 
manifold with boundary and 
$$
f:P\times M\to \Real
$$
a continuous function such that the $x$-derivative of $f$ 
(that is, the gradient of $f_p$, where $f_p(x)=f(p,x)$) 
is continuous on $P\times M$. We will be talking of $P$ as the {\em
  parameter space}. 

We denote by $\tauf:=\min_{p\in P} f_p$ the {\em min}-function of the
family $f$. The set $N\subset P\times M$ defined by
$$
N:=\{(p,x): f(p,x)=\tauf(x)\}
$$
is compact, and the slices
$$
N_x:=\{p\in P: (p,x)\in N\}
$$
are {\em upper semi-continuous}: for any $x\in M$ 
and any open neighborhood $UN_x\supset N_x$
there exists an open neighborhood $Ux\ni x$
such that for $x'\in Ux$,
$N_{x'}\subset UN_x$.

Next we show that if one can perturb each $x$ to increase $\tau$ then we can do so globally with a minimum speed.

\begin{lemma}\label{mintype}

Assume that for any $x\in M$, there exists a tangent vector $V_x\in
T_xM$ such that $L_{V_x}f_p>0$ for all $p\in N_x$. Then
\begin{itemize}
\item
for  some positive
$v$
there exists smooth vector {\em field} $V$ on $M$ such that $L_Vf_p\geq v>0$ in some
open vicinity of $N$, and

\item along the trajectories of $V$, the {\em min}-function
  $\tauf$ increases with speed at least $v$.
\end{itemize}
  \end{lemma}

\begin{proof} For any $x\in M$, we can extend the vector $V_x\in T_xM$
  to a smooth vector field on $M$ (which we still denote as $V_x$), 
such that $L_{V_x}f>0$ in some open vicinity $UN_x\times Ux$ of
$N_x\times\{x\}$. 
By compactness,
there exists a finite collection of points $\{x_i\}$ in $M$ such 
that the open sets $U_i:=UN_{x_i}\times Ux_i$ cover
  $N$ (and
  the open sets $Ux_i$ cover $M$), and $v>0$ such that $L_{V_i}f_p\geq v$ on $U_i$ 
 (here $V_i:=V_{x_i}$). Using a partition of unity we
  arrive at the first conclusion. The second conclusion is
  immediate.
%
%
\end{proof}


For $x\in M$ consider the intersection of the open half-spaces 
$$
H_x(p):=\{v\in T_xM: \langle df_p|_x, v\rangle >0\}
$$
over  all $p\in
N_x$. This is an open convex cone 
$$
C^o_x:=\bigcap_{p \in N_x} H_x(p)
$$
in $T_xM$.

Upper semicontinuity of $N_x$ implies lower semicontinuity of $C^o_x$;
for any $x \in M$ and any open set $V \subset TM$ intersecting $C^o_x$, there exists an open neighborhood $Ux \ni x$ such that for $x' \in Ux$, $C^o_{x'}$ intersects $V$.
In particular, if $C^o_x$ is non-empty, it remains such in a vicinity of
$x$.

Combining Lemmata \ref{Morse} and \ref{mintype} we obtain the
following 
\begin{corollary}\label{regular}
If the cones $C^o_x$ are non-empty over the level set 
$\tauf^{-1}(c)$, then $c$ is
topologically regular.
\end{corollary}

For general {\em min}-type functions this is essentially the best
possible condition for the regularity of the critical values. If the
functions $f_p$ are {\em quasi-convex}, i.e. have convex lower
excursion sets $\{f_p\leq c\}$, then Corollary~\ref{regular} can be considerable
strengthened. Thus, one can show (we will do it in a follow-up paper)
that a critical value is topologically regular, if for all points at
the level set, the intersection of the {\em closed} half-spaces is a
cone over a contractible base. This observation relies on stratified
Morse theory due to Goresky and MacPherson\cite{GMPh}, but is in a nutshell close to 
the elementary result used by Connelly in his work on the existence of continuous
``unlocking'' deformations of hard ball configurations, see
\cite{Connelly}. 




Corollary \ref{regular} implies that unless the level set of the
tautological function 
$\tau^{-1}(r)$
contains a point $x$ with
$C^o_x=\varnothing$, the homotopy type of $\confdb(n,r)$ is locally
constant at $r$. 

By Farkas' lemma, the emptiness of the cone $C^o_x$ implies that there exists a
finite collection of points $p_i\in N_x,
i=1,\ldots, I\leq \dim M+1$, and 
positive weights $w_i>0$ such that 
\begin{equation}\label{convex}
\sum_i w_i {df_{p_i}|_x}=0.
\end{equation}

\section{Critical points and stress graphs}


In our hard 
spheres
setting, the vanishing of the convex combination (\ref{convex})
has a clear geometric interpretation.

For $\vx \in \conf(n,r)$, 
define a {\em stress graph} of $\vx$ to be a graph embedded in $\R^d$ whose vertices are the points $x_1, \ldots, x_n$ and boundary points $y \in \d\body$ where $d(x_i,y)=r$ for some $i$. The edges are the pairs $\{x_i,x_j\}$ where $d(x_i,x_j)=2r$ and $\{x_i,y\}$ where $d(x_i,y) = r$. Each edge $k$ is assigned a positive weight $w_k$.
The points $x_i$ are referred to as {\em internal points} and the points $y$ are referred to as {\em boundary points}.
We 
interpret this graph as a system of {\em mechanical stresses}, with
(repulsive) forces acting on the endpoints of a segment $k$ equal to
$w_k$ times the unit vector in the direction of $k$.
Call the mechanical stresses acting on boundary points {\em boundary mechanical stresses}.
Call a connected component {\em trivial} if it consists of a single point.
Call $\vx$ trivial if $\Gamma(\vx)$ has no edges.

A stress graph is said to be {\em balanced} if it satisfies the following condition.
\begin{itemize}
\item The mechanical stresses at each internal point sum to zero.\footnote{This result is similar to the necessary conditions for ``locking'', 
see e.g. \cite{Connelly}.}
\item The boundary mechanical stresses on each connected component sum to zero.
\end{itemize}
Say that the configuration is $\vx$ is {\em balanced} if it has a balanced stress graph.

Call an internal point {\em isolated} if it is not in the boundary of any edges.
For each point $x_i$ call the intersection of the stress graph with the points on the sphere $d(x_i,x)=r$ {\em kissing points} of $x_i$.
Call a kissing point that is also a boundary point a {\em boundary kissing point}.

\begin{lemma}
  Assume that $\vx \in \conf(n,r)$ is balanced. Then
  \begin{enumerate}
    \item each non-isolated internal point is in the convex hull of its kissing points, and
    \item each non-trivial connected component is contained in the convex hull of its boundary kissing points.
  \end{enumerate}
\end{lemma}

Now consider $\vx \in \conf(n)$ that is a critical point of $\tau$ with critical value $r$.
In \eqref{convex}, the parameters $p_i$ correspond either to the
pairs of touching hard spheres, $d(x_{a_i}, x_{b_i})=2r$, or to the hard
sphere $x_{c_i}$ touching the boundary, $d(x_{c_i},y_i)=r$, at a point
$y_i\in\partial\body$.
Let $\Gamma(\vx)$ be the corresponding stress graph for $\vx$ with weights given by the coefficients $w_i$ in \eqref{convex}.

\begin{theorem} \label{thm:balanced}
  If $\vx \in \conf(n)$ is a critical point of $\tau$ with critical value $r$, then $\vx$ is balanced and nontrivial as a point in $\conf(n,r)$.
\end{theorem}

\begin{proof}
Consider $\vx \in \conf(n,r)$ and consider $\Gamma(\vx)$.
From \eqref{convex} it follows that the mechanical stresses at each internal point  sum to zero. 


The sum of mechanical stresses in each connected component equals the sum of mechanical stresses on internal points and the sum of external mechanical stresses. From \eqref{convex} and the first observation it follows that the boundary mechanical stresses on each connected component sum to zero.

Finally, since the sum in \eqref{convex} is nontrivial, 
$\vx$ is nontrivial.
\end{proof}

\section{Hard spheres in a box}
Consider now in more detail the case of hard spheres in a
rectangular box with sides $L:=L_1\leq  \ldots\leq L_d$, given, for
definitiveness sake, by
\[
\body=\{0\leq f_m \leq L_m, m=1,\ldots, d\}.
\] 
(Here $\{f_m\}$ is the orthonormal coordinate system on $\Real^d$.)

\subsection{Initial interval}

We 
show that Theorem~\ref{thm:balanced} implies a lower bound on the length of the
initial interval of values of $r$, where the homotopy type of
$\confdb(n,r)$ remains constant.

\begin{theorem}
  For the rectangular box $\body$, there are no critical values of
  $\tauf$ in $(0, L/2n)$, and therefore, 
\[
\confdb(n,r) \simeq \conf(n)
\]
for $r<L/2n$.
  
\end{theorem}

\begin{proof}
  Assume that $\vx \in \conf(n)$ is a critical value for $\tau$ with critical value $r$.
  Then by Theorem~\ref{thm:balanced}, $\vx$ is balanced and nontrivial.
  A connected component of $\Gamma(\vx)$ contains at most $n$ internal points and thus has diameter at most $2nr$.
  Since $\vx$ is nontrivial, it has at least one nontrivial connected component. It is contained in the convex hull of its boundary points, so it contains at least one boundary point. Since the boundary mechanical stresses of this connected component sum to zero, it must contain a pair of boundary points from opposing faces. Thus the diameter of this connected component is at least $L$.
Therefore $r \geq L/2n$.
\end{proof}

We remark that the balanced stress graph of minimal diameter is not necessarily a
segment, for non-rectangular boxes. For example, for the ``concave
triangle'', it is a cone over three points, see Figure \ref{fig:triangle}.

\begin{figure}[htb]
  \centering
  \includegraphics[height=2in]{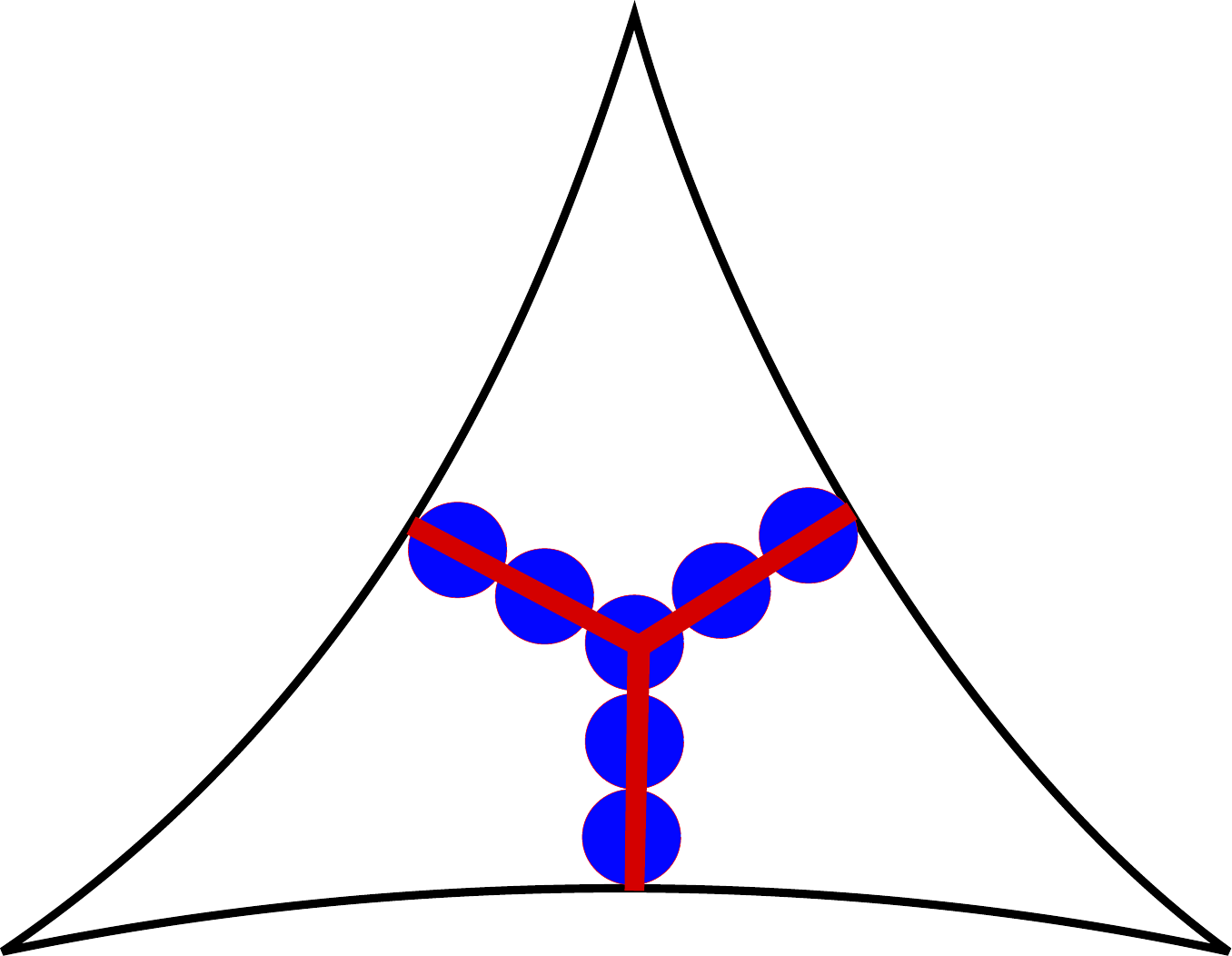}
  \caption{Minimal length stress graph}
  \label{fig:triangle}
\end{figure}






\subsection{First perestroika} \label{sec:perestroika}
A natural question is now to ask, whether there is a topology change
as $nr$ goes above the minimal length of the stress graph. We
concentrate in the rest of the note on the case of the rectangular box
with the shortest side of length $L$, and will investigate, whether
\[
i:\confdb(n,r')\to \confdb(n,r), r'=L/2n+\epsilon, r=L/2n-\epsilon
\]
is a homotopy equivalence, for small enough $\epsilon$.

We argue that it is not, by presenting
explicit 
nontrivial 
$(dn-n-d)$-cycles in $\ker(Hi)\subset
H_{dn-n-d}(\confdb(n,r'),\Int)$. 

Indeed, let $0<\epsilon<L/2n(n-1)$ (so that $(n-1)$ disks of radius $r'$
would fit within the box when arranged in a vertical column, and $n$
would not), 
and consider the
set $S_\epsilon$ of $n$-point configurations in $\body$ given by the conditions
\begin{itemize}
\item $x_1$ fixed is at distance $r'$ from the center of the face
  $\{f_1=0\}$,
\item $|x_{i+1}-x_i|=2r'$ for $i=1,\ldots n-1$, and 
\item $x_n$ is at distance $r'$ from the face $\{f_1=L_1\}$.
\end{itemize}
In other words, we consider the configurations for which the $n$ disks
touch each other and the opposite horizontal faces, forming a chain,
see Figure \ref{fig:chain}.

\begin{figure}[htb]
  \centering
  \includegraphics[height=2in]{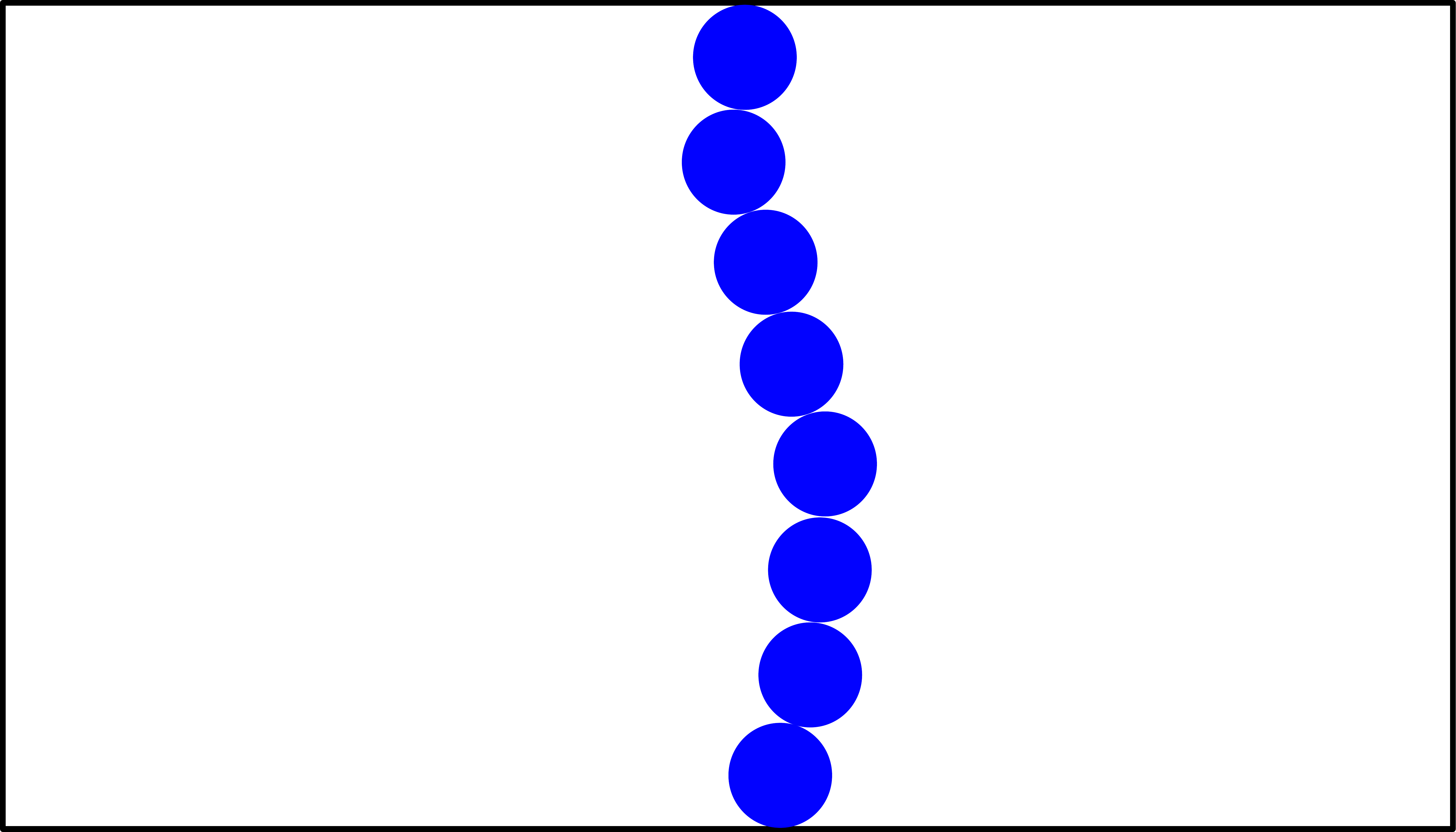}
  \caption{A configuration in $S_\epsilon$.}
  \label{fig:chain}
\end{figure}
An immediate computation shows that $S_\epsilon$ is diffeomorphic to a
$(nd-n-d)$-dimensional sphere. Orient it in some way, obtaining a
class $s\in H_{nd-n-d}(\confdb(n,r'))$.

Next 
we show that $s$ is nontrivial by constructing a cohomology class with which it has a nontrivial pairing.

Consider now the set $\Sigma$ of configurations in
$\body^n$
given by 
\begin{itemize}
\item all points $x_1,\ldots,x_n$ have the same coordinates $f_3,
  \ldots,f_d$;
\item all points
  $x_2,\ldots, x_n$ have the same coordinates $f_2,\ldots,f_d$;
\item the $f_1$ coordinates of $x_2,\ldots, x_n$ satisfy
$$
f_1(x_1)\geq r; f_1(x_{i+1})-f_1(x_i)\geq 2r, \mathrm{for\,}
i=2,\ldots, n-1; f_1(x_n)\leq L-r,
$$
and
\item $f_2(x_1)\leq f_2(x_2)$.
\end{itemize}

In other words, the configuration consists of $n-1$ vertically
aligned nonoverlapping 
$r$-disks constrained to have the same $(f_1,f_2)$-plane (with the
same $f_3,\ldots_d$ coordinates as
the disk $x_1$), see Figure \ref{fig:float}.

\begin{figure}[htb]
  \centering
  \includegraphics[height=2in]{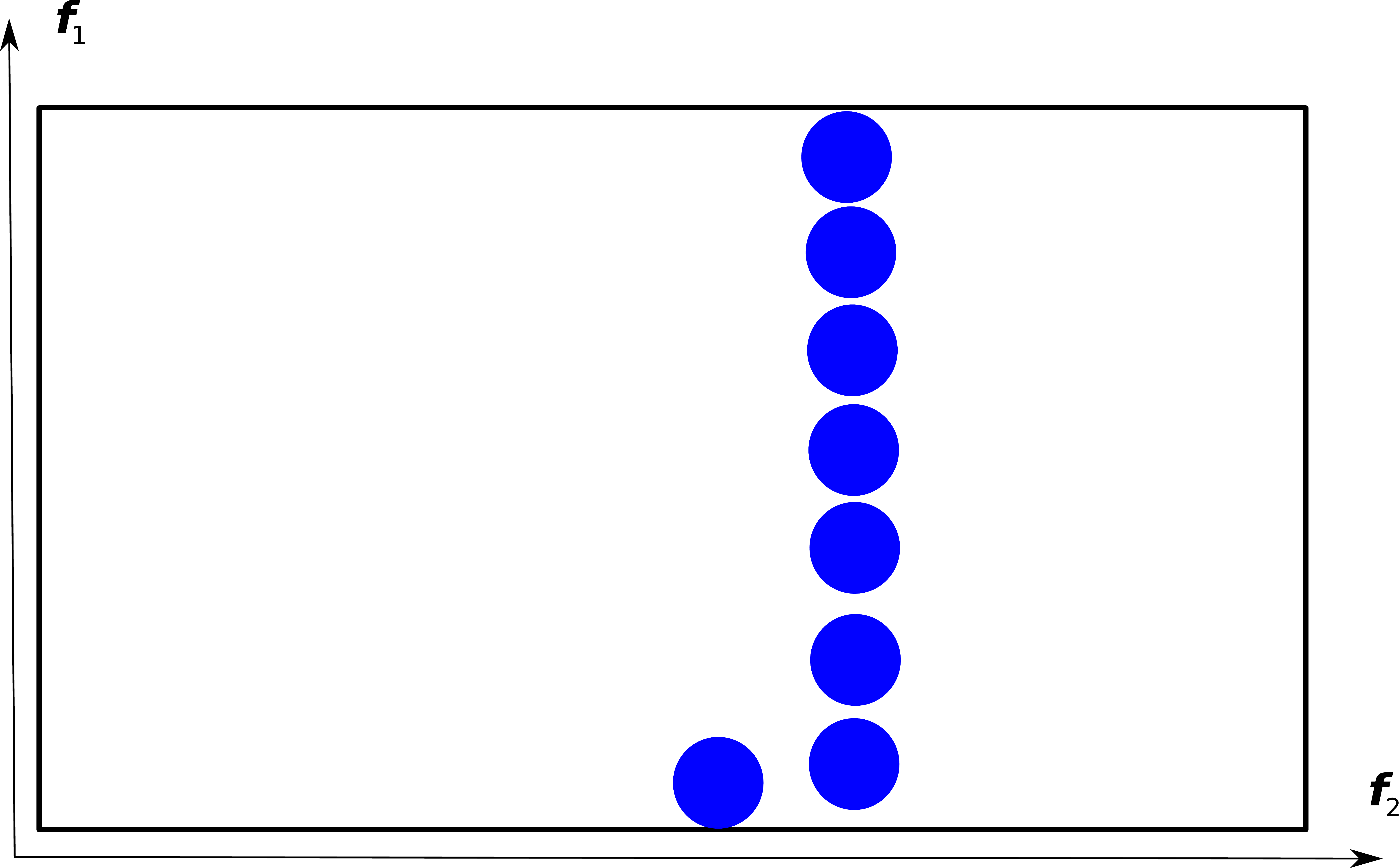}
  \caption{A configuration in $\Sigma$.}
  \label{fig:float}
\end{figure}

The conditions above are given by a finite collection of linear
equalities and inequalities, and therefore define a 
convex polyhedron of dimension $d+n$. The boundary of this
polyhedron is in $\body^n-\confdb(n,r')$, 
whence, upon orientation it defines 
a relative class $\sigma\in H_{n+d}(\body^n,
\body^n-\confdb(n,r'))$. 

We notice that the space $\body^n$ of $n$-tuples of 
points in $\body$ can be embedded into the $nd$-dimensional sphere
$S^{nd}$ (consider a large ball containing $\body^n$ and contract its boundary
to a point). By excision and the long exact sequence for a pair, 
$H_{n+d}(\body^n, \body^n-\confdb(n,r'))\cong H_{n+d}(S^{nd},
S^{nd}-\confdb(n,r')) \cong H_{n+d-1}(S^{nd}-\conf(n,r'))$. By Alexander duality the class
$\sigma$ can be identified with a class (which we still denote by
$\sigma$)
in
$H^{nd-n-d}(\confdb(n,r'))$. 

\begin{lemma}
The pairing between the classes $s$ and $\sigma$ is non-trivial: 
$s\cdot\sigma=\pm 1$.
\end{lemma}
\begin{proof}
Indeed, the manifolds $S_\epsilon$ and $\Sigma$ intersect 
transversally at a single point.
\end{proof}

As one can observe,
there exists a retraction of $S_{\epsilon}$ to a point staying within $\confdb(n,r)$, implying that 
the class
$s$ is in the kernel of $Hi$. Indeed, we first can reduce all the
distances between by shrinking the differences between the adjacent
chain centers so that
$$
x_{i+1}-x_i\mapsto \frac{tr-(1-t)r'}{r'} (x_{i+1}-x_i), t=[0,1];
i=1,\ldots,n-1
$$
and $x_1$ remains fixed 
(clearly, this homotopy keeps the configuration in
$\confdb(n,r)$). Then one can pull all the vectors $ (x_{i+1}-x_i)$ so
that they point vertically upwards\footnote{We think of $f_1$ as
  height.} 
(as not one was initially pointing downwards).

For each permutation $\pi$ of indices $1,\ldots,n$ one obtains
different classes $s_\pi$, and 
one can easily see that the pairing with the corresponding $n!$
classes $\sigma_\pi$ is non-degenerate (because corresponding $S_\epsilon$ and
$\Sigma$'s are all geometrically distinct). Hence, the rank of the kernel
of $Hi$ is at least $n!$.

We notice that the sphericity of the 
set
$S_\epsilon$ does not
depend on the fact that the stress graph is a chain. For the
configuration on Figure \ref{fig:triangle}, the corresponding 
set is diffeomorphic to
a sphere as well: it is just the corollary of the first
critical value coming from a {\em topologically Morse} critical point,
compare \cite{matov}.

\subsection{Betti numbers}

We can also compute how the Betti numbers change across the first threshold.
Set $ r_* = L / 2n$, and note that as the tautological function is
semi-algebraic for semi-algebraic regions, its critical values are
isolated. As the only balanced stress graphs in the case of a
rectangular domain are the chains spanning the shortest dimension, for
some small $\epsilon$ there are no other critical values in $(r_*-\epsilon, r_*+\epsilon)$.


It is well known \cite{Cohen} that the configuration space $\config(n)$ of $n$
(labeled) points in $\R^d$  has Poincar{\'e} polynomial
\begin{align*}
P(t) & := \sum_{i \ge 0} \beta_i t^i\\
	& = \prod_{i=1}^{n-1}  \left( 1 + i t^{d-1} \right)\\
	& = 1 + \dots + (n-1)! H_{n-1} t^{(n-2)(d-1)}+(n-1)!t^{(n-1)(d-1)},\\
\end{align*}
where $$H_{n-1} = \sum_{i=1}^{n-1} 1/i.$$

This tells us the Betti numbers of $\config(n , r_* -\epsilon)$, since we have already shown that $\config(n,r_*-\epsilon)$  is homotopy equivalent to 
$\config(n)$.  We wish to compute the Betti numbers of $\config(n, r_* + \epsilon)$.

Let $N = (n-1)(d-1)$.  As we shrink the disks across the critical value $r_* = L / 2n$,  to the configuration space  we attach $k \, n!$ cells of dimension $N$, where $k$ is the largest number such that $L_k = L$, whose boundaries are representatives for the homology classes $s$ defined in Section~\ref{sec:perestroika}.  Each of these cells either increments $\beta_N$ or decrements $\beta_{N-1}$.

The first observation is that $$\beta_i [ \config(n, r_* + \epsilon) ] = \beta_i [ \config(n, r_* - \epsilon) ] $$ for $ i \le N-2$.  
As (the proof to appear in a follow-up paper) one can show that
$\beta_i(\conf(n,r))=0$ for $i\geq N$ and $r>r_*$,
so in particular  $$\beta_i [ \config(n, r_* + \epsilon) ] = 0$$ for $i \ge N$.   
Thus $(n-1)!$ of the $N$-cells increase $\beta_N$.
This leaves only $\beta_{N-1}$ to compute.

Every $N$-cell that does not contribute to $\beta_N$ decreases $\beta_{N-1}$.  Since we know that $kn!$ cells are added, $(n-1)!$ of them contributing to $\beta_N$, we have
$$\beta_{N-1} [ \config(n, r_* + \epsilon) ] = \beta_{N-1} [\config(n, r_* - \epsilon)] + k \, n! - (n-1)! \; ,$$
and since  
\begin{displaymath}
  \beta_{N-1}[\conf(n,r_*-\epsilon)] = \left\{
    \begin{array}{ll}
      (n-1)! H_{n-1} & : d=2\\
      0 & : d \ge 3,
    \end{array}
  \right.  
\end{displaymath}
we have
\begin{displaymath}
  \beta_{N-1}[\conf(n,r_*+\epsilon)] = \left\{
    \begin{array}{ll}
      (H_{n-1}+ kn-1)(n-1)! & : d=2\\
      (kn - 1)(n-1)! & : d \ge 3.
    \end{array}
  \right.  
\end{displaymath}

\section{Concluding remarks}

In a future article we will discuss non-degeneracy of critical points, which is closely related to the question of making our necessary condition for a change in the topology sufficient.  We also discuss defining and computing the index of critical points, and especially investigate more of the asymptotic properties of $\conf(n,r)$ as $n \to \infty$.  In particular we obtain bounds on the rate of growth of Betti numbers.

An important special case for which little seems known is: What is the threshold radius $r=r(n)$ for connectivity of $\conf(n,r)$?  This is an important question physically, since for example ergodicity of any Markov process hinges on connectivity of the state space.  Diaconis, Lebeau, and Michel noted that $r \le c / n$ is sufficient to guarantee connectivity of $\config(n,r)$ \cite{DLM} and this is best possible for certain regions.  It would be interesting to know if connectivity of the configuration space ever extends into the thermodynamic limit, i.e.\ are there any bounding regions so that $\config(n,r)$ is connected for $r \le C n ^{-1/d}$ and some constant $C > 0$? 

\section*{ {\bf Acknowledgments} } 
We thank AIM and for hosting the workshop on, ``Topological complexity
of random sets'' in August 2009, where we started discussing some of
these problems.  Y.B. was supported in part by the ONR grant
00014-11-1-0178. 
M.K. thanks IAS for hosting him this year
and Robert MacPherson for several helpful conversations. 

\bibliographystyle{plain}
\bibliography{hard}

\end{document}